\documentclass[12pt]{article}
\usepackage[utf8]{inputenc}

\usepackage{amssymb,amsmath,amsthm,epsfig,tabularx}
\usepackage{subfigure}
\usepackage{float}
\usepackage{url,epsf,graphicx,tikz}
\usetikzlibrary{decorations.markings}

\newtheorem{theorem}{Theorem}
\newtheorem{proposition}[theorem]{Proposition}
\newtheorem{lemma}[theorem]{Lemma}
\newtheorem{corollary}[theorem]{Corollary}

\newtheorem{conjecture}[theorem]{Conjecture}

\newcommand{\dist}{{\rm dist}}
\newcommand{\SJ}{{\rm SJ}}

\begin{document}

\title{On the minimum value of sum-Balaban index}

\author
{ Martin Knor\thanks{Slovak University of Technology in Bratislava,
Faculty of Civil Engineering, Department of Mathematics, Bratislava,
Slovakia. E-Mail: \texttt{knor@math.sk}},
\quad Jaka Kranjc
\thanks{Faculty of Information Studies, Novo mesto, Slovenia. E-Mail:
\texttt{jaka.kranjc@fis.unm.si}},
\quad Riste
\v{S}krekovski\thanks{FMF, University of Ljubljana \& Faculty of Information Studies, Novo mesto \& FAMNIT, University of Primorska, Koper, Slovenia.  
E-Mail: \texttt{skrekovski@gmail.com}},\quad Aleksandra Tepeh\thanks{Faculty
of Information Studies, Novo mesto \& Faculty of Electrical
Engineering and Computer Science, University of Maribor, Slovenia.
E-Mail: \texttt{aleksandra.tepeh@gmail.com}} }

\maketitle

\begin{abstract}
We consider extremal values of sum-Balaban index among graphs on $n$ vertices. We determine that the upper bound for the minimum value of the sum-Balaban
index is at most $4.47934$ when $n$ goes to infinity.
For small values of $n$ we determine the extremal graphs and we observe
that they are similar to dumbbell graphs, in most cases having one extra
edge added to the corresponding extreme for the usual Balaban index.
We show that in the class of balanced dumbbell graphs, those with clique
sizes $\sqrt[4]{\sqrt 2\log\big(1+\sqrt 2\big)}\sqrt n+o(\sqrt n)$ have
asymptotically the smallest value of sum-Balaban index.
We pose several conjectures and problems regarding this topic.
\end{abstract}


\tikzset{My Style/.style={draw, circle, fill=black,scale=0.3}} 

%

%
\section{Introduction}

In this paper we consider simple and connected graphs.
Denote by $V(G)$ and $E(G)$  the vertex and edge sets of a given graph $G$,
respectively.
Let $n=|V(G)|$ and $m=|E(G)|$.
For vertices $u,v\in V(G)$, by $\dist_G(u,v)$ (or shortly just $\dist(u,v)$)
we denote the distance from $u$ to $v$ in $G$, and by $w(u)$ we denote
the {\em transmission} (or the {\em distance}) of $u$, defined as
$w(u)=\sum_{x\in V(G)}d_G(u,x)$.

{\em Balaban index} $J(G)$ of a connected graph $G$, defined as
$$
J(G)=\frac m{m-n+2}\sum_{uv\in E(G)}\frac 1{\sqrt{w(u)\cdot w(v)}},
$$
was introduced in early eighties by Balaban \cite{B1,B2}.
Later, Balaban et al. \cite{Bsum} (and independently also Deng \cite{deng-sum})
proposed a derived measure, namely the {\em sum-Balaban index} $\SJ(G)$ for
a graph $G$:
$$
\SJ(G)=\frac m{m-n+2}\sum_{uv\in E(G)}\frac 1{\sqrt{w(u)+w(v)}}.
$$

Similarly as Balaban index, also sum-Balaban was used in various quantitative
structure-property relationship (QSPR) and quantitative structure activity
relationship (QSAR) studies.
Regarding mathematical properties, there are several results known for
sum-Balaban index, but they mainly pertain to trees, and graphs containing
only one or two cycles.

It was shown by Deng \cite{deng-sum} and Xing et al. \cite{xing} that for
a tree $T$ on $n$ vertices, $n\geq 2$,
$$
\SJ(P_n)\leq \SJ(T)\leq \SJ(S_n)
$$
with left (right) equality if and only if $T = P_n$ ($T = S_n$),
where $P_n$ is the path on $n$ vertices and $S_n$ is the star on $n$
vertices.
The authors in \cite{xing} also determined trees with the second-largest,
and third-largest as well as the second-smallest, and third-smallest
sum-Balaban indices among the $n$-vertex trees for $n\geq 6$. In \cite{KST-SJ-trees} alternative proof for the above results and further
ranking up to seventh maximum sum-Balaban index was presented.
In \cite{you-han} the authors investigated the maximum sum-Balaban index
of trees with given diameter, and in \cite{pism} the extremal graph which
attains the maximum sum-Balaban index among trees with given vertices
and maximum degree were determined.

Unicyclic graphs on $n$ vertices with the maximum sum-Balaban index were
considered in \cite{wrong}, and bicyclic graphs were studied in
\cite{ChenDeh,fang}.
For various upper and lower bounds of general graphs in terms of some other
parameters (such as the maximum degree, number of edges, etc.) see \cite{deng-sum}
and \cite{xing}, and for recent results on $r$-regular graphs see \cite{lei}.

Maximal values of Balaban and sum-Balaban index in more general setting
were explored in \cite{KSTmsb}. On the other hand, finding the minimum value of sum-Balaban index
among $n$-vertex graphs is a rather untractable problem.
We find it natural to attack this problem in a similar fashion as in the
case of Balaban index, so in general we follow the steps of \cite{KKST}.
For small values of $n$ we determine the extremal graphs and we observe
that they are similar to dumbbell graphs, in most cases having one extra
edge added to the corresponding extreme for the usual Balaban index.
We show that in the class of balanced dumbbell graphs, those with clique
sizes $\sqrt[4]{\sqrt 2\log\big(1+\sqrt 2\big)}\sqrt n+o(\sqrt n)$ have
asymptotically the smallest value of sum-Balaban index.
Recall that for Balaban index, the coefficient in front of $\sqrt n$ is
$\sqrt[4]{\pi/2}$, see \cite{KKST}.
Using a computer we find dumbbell-like graphs with slightly smaller
sum-Balaban index values.
We also pose several conjectures and problems regarding this topic.

%
%
%

\section{Two simple lower bounds on sum-Balaban index}

We begin by stating two simple lower bounds for sum-Balaban index
in the class of graphs on $n$ vertices.
Note that these two claims correspond to Theorems~4 and~8 for Balaban index,
see~\cite{KKST}.

\begin{theorem}
\label{thm:lower}
Let $G$ be a graph on $n$ vertices, $n\ge 4$.
Then
$$
\SJ(G)\ge 2\sqrt{\frac n{n-1}}\,.
$$
\end{theorem}

\begin{proof}
Suppose that $G$ has $m$ edges.
Since $n\ge 4$, we have
\begin{equation}
\label{eq:mn_lower}
\frac{m}{m-n+2}\ge\frac{2n}m.
\end{equation}
For every vertex $v\in V(G)$, it holds
$w(v)\le 1+2+\dots +(n{-}1)=\frac{n^2-n}2$.
Hence, for every $u,v\in V(G)$ we have 
\begin{equation}
\label{eq:w_lower}
 \frac{1}{\sqrt{w(u) + w(v)}}\ge  \frac{1}{\sqrt{n^2-n}}. 
\end{equation}
Since $G$ has $m$ edges, using (\ref{eq:mn_lower}) and (\ref{eq:w_lower}) we
obtain
$$
\SJ(G)  = \displaystyle \frac{m}{m{-}n{+}2}\sum_{uv\in E(G)}\frac 1{\sqrt{w(u){+}w(v)}}
\ge \frac{2n}{m}\cdot\frac{m}{\sqrt{n^2{-}n}} \nonumber =   2\sqrt{\frac n{n-1}}.
$$
\end{proof}

For large values of $n$ we present a better lower bound on the sum-Balaban index.

\begin{theorem}
\label{thm:lower_big}
Let $G$ be a graph on $n$ vertices, where $n$ is big enough.
Then
$$
\SJ(G)\ge 4+o(1).
$$
\end{theorem}

\begin{proof}
Let $f(n)$ be a function that represents the number of edges in extremal
graphs on $n$ vertices.
Since our graphs are connected, we have $m\ge n-1$, that is, $f(n) \in \Omega(n)$.
Now we split the proof into two cases according to the behaviour of $f(n)$:

\begin{itemize}
\item{Case 1:}
{\it $f(n) \not\in \Theta(n)$}.
This means there is a subsequence $\{f(n_i)\}_{i=1}^{\infty}$, such
that for every constant $k$ we have $f(n_j)>k n_j$ for all $j$ big enough.
From (\ref{eq:w_lower}), for $n$'s in this subsequence we get
\begin{equation}
\label{eq:lower_c1}
\SJ(G)\ge \frac{m}{m-n+2}\cdot\frac m{\sqrt{n^2-n}}
\sim\frac{f(n)}{f(n)}\cdot\frac{f(n)}{n}
\sim f(n)\cdot n^{-1}.
\end{equation}
However, Corollary~{\ref{cor:dumbbell}} gives a dumbbell graph $D$
on $n$ vertices with sum-Balaban index smaller than $5$.
Hence $5>\SJ(D)$ which means that $f(n)$, the function representing the number of 
edges in extremal graphs, must satisfy $5> f(n)n^{-1}$ by (\ref{eq:lower_c1}).
Hence $5n>f(n)$. But this contradicts the properties of $\{f(n_i)\}_{i=1}^{\infty}$, and so
there is not a subsequence as required in this case.

\item{Case 2:}
{\it $f(n) \in \Theta(n)$.
This means that there are positive constants $c_1$ and $c_2$, such that
for large $n$ we have $c_1n\le f(n)\le c_2n$.}
Fix $n$ big enough.
Then there is $c$ ($=c(n)$) such that $c_1\le c\le c_2$ and $f(n)=cn$.
From (\ref{eq:w_lower}) we get
\begin{equation}
\label{eq:lower_big}
\SJ(G)\ge
\frac{m}{m-n+2}\cdot\frac{m}{\sqrt{n^2-n}}
\sim\frac{cn}{(c-1)n}\cdot\frac{cn}{n}
=\frac{c^2}{c-1}.
\end{equation}
Notice that the right-hand side of (\ref{eq:lower_big}) is minimal for $c=2$.
Substituting $c=2$ in (\ref{eq:lower_big}) we obtain $\SJ(G)\ge 4+o(1)$.
\end{itemize}
\end{proof}

By Theorem~{\ref{thm:lower_big}}, the asymptotic lower bound for $\SJ(G)$
is $4$.
Let us mention that nanotubes of type $(k,l)$ (regardless if they are
open or not) have asymptotic value of sum-Balaban index
$\frac{9\sqrt 2}{2}\sqrt{k+l}\cdot\log(1+\sqrt 2)$, see \cite{AKS1, AKS2}.
However, in the sequel we show that there are graphs with even smaller
value of Balaban index.

%
%
%

\section{Extremes for small number of vertices}

By the proof of Theorem~{\ref{thm:lower_big}}, one would expect that
a graph with the minimum sum-Balaban index will have $\Theta(n)$ edges
and vertices $v$ with big value of $w(v)$. Hence, a path with two complete graphs attached to the end-vertices
of the path, so called {\em dumbbell graph}, is a good candidate
for an extremal graph. This idea is supported by the list of extremal graphs for $n\le 10$,
see Figure~\ref{fig:dumb10}, so we devote this section to dumbbell graphs.
Some of the graphs on the figure contain a dotted edge.
By removing such edge we obtain the graph with the minimum value of
Balaban index.

We restrict ourselves to $n\le 10$ as the realm of graphs is getting huge for bigger $n$. Perhaps, with a little more powerful computer resources the cases $n=11$ and $12$ could  be easily tractable. However, in Figure~\ref{fig:dumb11} we present graphs with a potential
to be the extremes for $n\in [11,14]$. These graphs were obtained by restricting the space of graphs of order $n$
(to maximal degree up to $5$, for \(n = 11\) and \(n = 12\) to graphs containing at least $n$ and at most $22$ edges, and for  \(n = 13\) and \(n = 14\) to graphs with at least \(15\) and at most \(20\) edges).

%
%

\section{Bounds for balanced dumbbell graphs}

Here we consider the lower bound of sum-Balaban index among balanced
dumbbell graphs in a similar way as it was done in \cite{KKST} for Balaban
index.
Reason for this is that these graphs are simple to define and deal with and
in some cases they are extremal, see Figures~\ref{fig:dumb10}
and~\ref{fig:dumb11}.
We believe that balanced dumbbell graphs asymptotically attain the lower
bound.

Let $K_a$ and $K'_{a'}$ be two disjoint complete graphs on $a$ and $a'$
vertices, respectively.
We always assume $a\le a'$.
Further, let $P_b$ be a path on $b$ vertices $(v_0,v_1,\dots,v_{b-1})$
disjoint from the cliques.
The {\em dumbbell graph} $D_{a,b,a'}$ is obtained from
$K_a\cup P_b\cup K'_{a'}$ by joining all vertices of $K_a$ with $v_0$ and
all vertices of $K'_{a'}$ with $v_{b-1}$.
Thus, $D_{a,b,a'}$ has $a+b+a'$ vertices.
In the case when $a=a'$, we call a graph a {\em balanced dumbbell graph} and
we denote it by $D_{a,b}$.

Also for Balaban index, balanced dumbbell graphs are close to extremal
graphs.
However, the cliques and paths in balanced dumbbell graphs achieving the
minimum value of Balaban index have different sizes from those, which achieve
the minimum value of sum-Balaban index.
To derive these sizes, we use a lemma from \cite{KKST}.

\begin{lemma}
\label{lema:w}
Let $u$ be a vertex of $K_a$ (or $K'_a$) and let $v_i$ be the $i$-th
vertex of $P_b$, where $K_a$, $K'_a$ and $P_b$ are parts of
the balanced dumbbell graph $D_{a,b}$ as in the definition.
Then
\begin{align*}
w(u) &= \tfrac{b^2}2 + ab + \tfrac b2 + 2a-1, \qquad \mbox{and}\\
w(v_i) &= \tfrac{i^2}2 + \tfrac{(b-i)^2}2 + ab - \frac b2 + i + a.
\end{align*}
\end{lemma}

Next result gives an upper bound for the minimum value of sum-Balaban index
in the class of balanced dumbbell graphs.

\begin{proposition}
\label{lem:hat}
Let $c$ be a positive constant.
Further, let $D_{a,b}$ be a balanced dumbbell graph on $n$ vertices, where
$a\sim c\sqrt{n}$. Then
$$
\SJ(D_{a,b})\in\Theta(1).
$$
\end{proposition}
\begin{proof}
Since $a\sim c\sqrt{n}$, we have $b\sim n$.
Therefore, $w(u)\sim\frac{b^2}2=w^*(u)$ if $u$ is a
vertex of $K_a$ or $K'_a$, while
$w(v_i)\sim\frac{i^2}2+\frac{(b-i)^2}2=w^*(v_i)$ for $v_i\in V(P_b)$,
by Lemma~{\ref{lema:w}}.
Since $\frac{b^2}4\le \frac{i^2}2+\frac{(b-i)^2}2\le\frac{b^2}2$,
for every edge $xy$ we have
$$
\frac b{\sqrt 2}\le\sqrt{w^*(x) + w^*(y)}\le b.
$$
Hence, for every edge $xy$ there exist $c^b_{xy} \in [\frac{1}{\sqrt 2}, 1]$
such that
$\sqrt{w^*(x) + w^*(y)}=c^b_{xy} b$.
Then
$$
\frac 1{\sqrt{w(x) + w(y)}}\sim\frac 1{c^b_{xy}b}.
$$
Since $D_{a,b}$ has $2\binom{a+1}2+b-1$ edges, we have
$m=a^2+a+b-1\sim a^2+b$.
Thus, analogously as above we can get
$$
\sum_{xy\in E(D_{a,b})}\frac 1{\sqrt{w(x) + w(y)}}\sim
\frac{a^2+b}{c^b\cdot b},
$$
where $c^b$ is some value such that $c^b \in [\frac{1}{\sqrt 2}, 1]$.
Finally, $m-n+2=a^2-a+1\sim a^2$.
Hence,
\begin{equation} \label{eq:hat}
  \begin{split}
       \SJ(D_{a,b}) & =  \frac{m}{m-n+2}\sum_{uv\in E(D_{a,b})}\frac 1{\sqrt{w(u) + w(v)}}  \sim  \frac{a^2+b}{a^2}\cdot \frac{a^2+b}{c^b\cdot b} \\
                           & =  \frac 1{c^b}\bigg[\frac{a^2}b + 2 + \frac b{a^2}\bigg],
  \end{split}
\end{equation} 
where $c^b$ is a value such that $\frac 1{\sqrt 2}\le c^b\le 1$.
Recall that $a\sim c\sqrt{n}$.
Since all terms in brackets of the second line of (\ref{eq:hat}) are
in $\Theta(1)$, we conclude $\SJ(D_{a,b})\in\Theta(1)$.
\end{proof}

In what follows we will need the following result from analysis.

\begin{proposition}
\label{lemmaA}
Let $b$ be a positive integer.
Then as $b\to\infty$, we have
\begin{equation}
\label{eq:whatever}
\sum_{i=0}^b\frac 1{\sqrt{i^2+(b-i)^2}}\sim \sqrt 2\log\big(1+\sqrt 2\big).
\end{equation}
\end{proposition}
\begin{proof}
Since $g(x)=1/\sqrt{x^2+(1-x)^2}$ is a continuous and concave function on the closed interval [0,1], it has the Riemann integral, which implies that
$$
\int_0^1 g(x) dx \sim \frac 1b\sum_{i=0}^b\frac 1{\sqrt{(\frac ib)^2+(\frac{b-i}b)^2}}=\sum_{i=0}^b\frac 1{\sqrt{i^2+(b-i)^2}}.
$$
Since $\int_0^1 g(x) dx=\sqrt 2\log\big(1+\sqrt 2\big)$, we obtain the desired result.
\end{proof}

In the next lemma we evaluate the contribution of the  edges of the path 
to the sum-Balaban index. 

\begin{lemma}
\label{lemmaB}
For a balanced dumbbell graph $D_{a,b}$ the following holds
$$
\sum_{i=0}^{b-2}\frac 1{\sqrt{w(v_i) + w(v_{i+1})}}
\sim \sum_{i=0}^b\frac 1{\sqrt{i^2+(b-i)^2}}.
$$
\end{lemma}
\begin{proof}
Let $v_iv_{i+1}$ be an edge of $P_b$.
Denote $w_i^+=\max\{w(v_i),w(v_{i+1})\}$ and
$w_i^-=\min\{w(v_i),w(v_{i+1})\}$.
Then $\frac 1{\sqrt{2w_i^+}}\le\frac 1{\sqrt{w(v_i) + w(v_{i+1})}}\le
\frac 1{\sqrt{2w_i^-}}$.
Therefore
$$
\sum_{i=0}^{b-2}\frac 1{\sqrt{2w_i^+}}\le
\sum_{i=0}^{b-2}\frac 1{\sqrt{w(v_i) + w(v_{i+1})}}
\le \sum_{i=0}^{b-2} \frac 1{\sqrt{2w_i^-}}.
$$
We have
\begin{align*}
\sum_{i=0}^{b-2}\frac 1{\sqrt{2w_i^+}}
= & \frac 1{\sqrt{2w(v_0)}}+\frac 1{\sqrt{2w(v_1)}}+\dots
+\frac 1{\sqrt{2w(v_{\lfloor\frac b2\rfloor-1})}}
+\frac 1{\sqrt{2w(v_{\lfloor\frac b2\rfloor+1})}}\\
& +\frac 1{\sqrt{2w(v_{\lfloor\frac b2\rfloor+2})}}
+\dots+\frac 1{\sqrt{2w(v_{b-1})}}\\
\sim & \sum_{i=0}^b\frac 1{\sqrt{i^2+(b-i)^2}}
-\frac 1{\sqrt{\lfloor\frac b2\rfloor^2+(b-\lfloor\frac b2\rfloor)^2}}
-\frac 1{\sqrt{b^2+0^2}}.
\end{align*}
Note that $\sum_{i=0}^b\frac 1{\sqrt{i^2+(b-i)^2}}$ is of order $\Theta(1)$
by Proposition~\ref{lemmaA}.
Also notice that the two isolated terms in the above expressions are of order
$O(n^{-1})$.
Therefore,
$\sum_{i=0}^{b-2}\frac 1{\sqrt{2w_i^+}}\sim
\sum_{i=0}^b\frac 1{\sqrt{i^2+(b-i)^2}}$.
Analogously we get
\begin{align*}
\sum_{i=0}^{b-2}\frac 1{\sqrt{2w_i^-}}
= & \sum_{i=0}^b\frac 1{\sqrt{2w(v_i)}}
-\frac 1{\sqrt{2w(v_0)}}+\frac 1{\sqrt{2w(v_{\lfloor\frac b2\rfloor})}}
-\frac 1{\sqrt{2w(v_{b-1})}}\\
&-\frac 1{\sqrt{2w(v_b)}}
\sim \sum_{i=0}^b\frac 1{\sqrt{i^2+(b-i)^2}}.
\end{align*}
This establishes the lemma.
\end{proof}

Now we can prove the main result of the paper.

\begin{theorem}
\label{thm:dumbbell}
Let $D_{a,b}$ be a balanced dumbbell graph on $n$ vertices with the smallest possible value of sum-Balaban index.
Then $a$ and $b$ are asymptotically equal to
$\sqrt[4]{\sqrt 2\log\big(1+\sqrt 2\big)}\sqrt n$ and $n$, respectively.
That is, $a=\sqrt[4]{\sqrt 2\log\big(1+\sqrt 2\big)}\sqrt n+o(\sqrt n)$
and $b=n-o(n)$.
\end{theorem}

\begin{proof}
Let $D_{a,b}$ be a balanced dumbbell graph on $n$ vertices with the minimum
value of sum-Balaban index.
By Proposition~{\ref{lem:hat}}, $\SJ(D_{a,b})\in O(1)$.
We study the behaviour of $a=a(n)$ in the following two claims.
First notice that since $a(n)<n$, we have $a\in O(n)$.
Therefore if $a\in\Omega(n)$, then $a\in\Theta(n)$.

\bigskip
\noindent
{\bf Claim 1.} {\em It holds $a\in o(n)$.} 

\medskip
\noindent
Suppose that the claim is false.
Then there is a subsequence $\{a(n_i)\}_{i=1}^{\infty}$ which is
in $\Theta(n)$.
By Lemma~{\ref{lema:w}}, we have $w(u)\sim\frac{b^2}2+ab+2a=w^*(u)$,
where $u$ is a vertex of $K_a$ or $K'_a$, and
$w(v_i)\sim\frac{i^2}2+\frac{(b-i)^2}2+ab+a=w^*(v_i)$ for $v_i\in V(P_b)$.
Since $\frac{b^2}4\le \frac{i^2}2+\frac{(b-i)^2}2\le\frac{b^2}2$,
for every vertex $x$ it holds
$$
\frac{b^2}4+ab+a\le w^*(x)\le\frac{b^2}2+ab+2a.
$$
Consequently, for every edge $xy$ we have
$$
\sqrt{\frac{b^2}2+2ab+2a}\le\sqrt{w^*(x) + w^*(y)}\le\sqrt{b^2+2ab+4a}.
$$
Hence, for every edge $xy$ there exist values $c^b_{xy} \in [\frac 12, 1]$
and  $c^a_{xy}\in [2,4]$ such that
$\sqrt{w^*(x) + w^*(y)}=\sqrt{c^b_{xy}b^2+2ab+c^a_{xy}a}$.
Then
$$
\frac 1{\sqrt{w(x) + w(y)}}\sim
\frac 1{\sqrt{c^b_{xy}b^2+2ab+c^a_{xy}a}},
$$
and analogously as in the proof of Proposition~{\ref{lem:hat}} we get
$$
\sum_{xy\in E(D_{a,b})}\frac 1{\sqrt{w(x) + w(y)}}\sim
\frac{a^2+b}{\sqrt{c^b b^2+2ab+c^a a}},
$$
where $c^b$ and $c^a$ are some values such that $c^b \in [\frac 12, 1]$
and  $c^a\in [2,4]$.
Estimating $m/(m-n+2)$ analogously as in (\ref{eq:hat}) we get
\begin{equation}
\label{eq:case1}
\SJ(D_{a,b})\sim \frac{a^2+b}{a^2}\cdot
\frac{a^2+b}{\sqrt{c^b b^2+2ab+c^a a}}.
\end{equation}
Since $b\le n$, the numerator in (\ref{eq:case1}) is of order $n^4$,
while the denominator is of order at most $n^3$.
This gives that for our sequence of $n$'s, $\{n_i\}_{i=1}^{\infty}$,
$J(D_{a,b})\to\infty$.
However, by Proposition~{\ref{lem:hat}}, for the very same subsequence
$\{n_i\}_{i=1}^{\infty}$ we have already derived that $J(D_{a,b})\in O(1)$,
which is a contradiction that establishes the claim.
\bigskip

With the next claim we go further and determine the asymptotic order of $a(n)$.

\bigskip
\noindent
{\bf Claim 2.} {\em It holds $a\in \Theta(\sqrt{n})$.} 

\medskip
\noindent
By Claim 1, for every positive constant $k$ we have $a(n)<kn$ for all $n$
big enough.
Since $2a+b=n$, we get $b>n-2kn$ for all $n$ big enough and consequently
$b\in\Theta(n)$.
We proceed analogously as in Case~1.
By the property of $\{a(n)\}_{n=1}^{\infty}$, we get
$w(u)=\frac{b^2}2+ab+\frac b2+2a-1\sim\frac{b^2}2$ if $u$ is a vertex of
$K_a$ or $K'_a$, while
$w(v_i)=\frac{i^2}2+\frac{(b-i)^2}2+ab-\frac b2+i+a\sim\frac{i^2}2+\frac{(b-i)^2}2$
for $v_i\in V(P_b)$.
Hence, analogously as in the proof of Proposition~{\ref{lem:hat}} we get
\begin{equation}
\label{eq:case2}
\SJ(D_{a,b})\sim \frac{a^2+b}{a^2}\cdot \frac{a^2+b}{c^b\cdot b}
=\frac 1{c^b}\bigg[\frac{a^2}b+2+\frac b{a^2}\bigg],
\end{equation}
where $c^b$ is a value such that $\frac 1{\sqrt 2}\le c^b\le 1$.
Recall that $a=a(n)$.
Then the terms in brackets of (\ref{eq:case2}) are in
$\Theta\big(\frac{a^2(n)}n\big)$, $\Theta(1)$,
$\Theta\big(\frac n{a^2(n)}\big)$, respectively, and the order of
$\SJ(D_{a,b})$ is the maximum order of these three terms.
Since the second term is in $\Theta(1)$, we have
$\SJ(D_{a,b})\in \Omega(1)$. But in the case $J(D_{a,b}) \in \Theta(1)$
we have $a^2(n)\in O(n)$ from the first term and $a^2(n)\in \Omega(n)$
from the third term.
This gives $a(n)\in \Theta(\sqrt{n})$, which establish the claim. 

\bigskip

By Claim~2, there are positive constants $c_1$ and $c_2$ such that
$c_1\sqrt n\le a\le c_2\sqrt n$ for each large enough $n$.
Hence, for $n$ big enough $a=c\sqrt n$, where $c=c(n)$ is in $[c_1,c_2]$,
and $b\sim n$. In the rest of the proof we determine $c$ ($=c(n)$).
In order to do this, we need a more precise calculation of $\SJ(D_{a,b})$.

Let $uv$ be an edge of $D_{a,b}$.
Then it is either in one of the two complete graphs on $a+1$ vertices, in
which case $w(u)\sim w(v)\sim \frac{b^2}2$, or it is an edge of the path
$P_b$, say $u=v_i$ and $v=v_{i+1}$, in which case
$w(v_i)\sim\frac 12[i^2+(b-i)^2]$ and
$w(v_{i+1})\sim \frac 12[(i+1)^2+(b-i-1)^2]$.
There are $2\binom{a+1}2$ edges of the first type and they contribute
$(a^2+a)\cdot\frac 1b\sim\frac{a^2}b$ to
$\sum_{uv\in E(D_{a,b})}\frac 1{\sqrt{w(u) + w(v)}}$.
By Lemma~\ref{lemmaB} and Proposition~{\ref{lemmaA}}, the contribution
of edges of $P_b$ is
$$
\sum_{i=0}^{b-2}\frac 1{\sqrt{w(v_i)+ w(v_{i+1})}}
\sim \sqrt 2\log\big(1+\sqrt 2\big).
$$
Denote $Q=\sqrt 2\log\big(1+\sqrt 2\big)$.
Since $m\sim a^2+b$ and $m-n+2\sim a^2$, we get
$$
\SJ(D_{a,b})\sim\frac{a^2+b}{a^2}\bigg(\frac{a^2}b+Q\bigg).
$$
Recall that $a\sim c\sqrt n$ while $b\sim n$.
Hence,
\begin{equation}
\label{eq:J(D)}
\SJ(D_{a,b})\sim\frac{c^2+1}{c^2}\cdot(c^2+Q)
=c^2+1+Q+\frac Q{c^2}.
\end{equation}
Now setting $c^2=x$ and differentiating the above expression we see that
$\SJ(D_{a,b})$ is minimum if $c^2=\sqrt Q$, that is if
$c=\sqrt[4]{\sqrt 2\log\big(1+\sqrt 2\big)}$.
\end{proof}

Observe that for Balaban index we have an analogue of
Theorem~{\ref{thm:dumbbell}}, but the value of $c$ is different.
Balanced dumbbell graphs with the minimum value of Balaban index have
$a\sim c\sqrt n$, where $c=\sqrt[4]{\pi/2}\doteq 1.1195$, see \cite{KKST},
while those with the minimum value of sum-Balaban index have
$a\sim c\sqrt n$, where $c\doteq 1.0566$, by Theorem~{\ref{thm:dumbbell}}.

Substituting $c=\sqrt[4]{\sqrt 2\log\big(1+\sqrt 2\big)}$ in (\ref{eq:J(D)}),
Theorem~{\ref{thm:dumbbell}} yields the following corollary.

\begin{corollary}
\label{cor:dumbbell}
Let $D$ be a balanced dumbbell graph on $n$ vertices, where $n$ is big enough,
with the minimum value of sum-Balaban index.
Then
$$
\SJ(D)\doteq 4.47934.
$$
\end{corollary}

Comparing Corollary~{\ref{cor:dumbbell}} with the lower bound
presented in Theorem~{\ref{thm:lower_big}}, we see that the asymptotic
value of sum-Balaban index for optimum balanced dumbbell graph is only
about 1.12 times higher than our lower bound.
Our expectation is that the optimal balanced dumbbell graph is not much
different from the optimal dumbbell graph.
Namely, we have the following conjectures.

\begin{conjecture}
\label{C1}
Among all dumbbell graphs $D_{a,b,a'}$ on at least $14$ vertices, the minimum value
of sum-Balaban index is achieved for one with $a'=a$ or $a'=a+1$.
\end{conjecture}

The reason for the assumption  that $n \geq 14$ in the above conjecture is that $D_{2,7,4}$ has the lowest sum-Balaban index among all dumbbell graphs on 13 vertices.

\begin{conjecture}
\label{C2}
Among all dumbbell graphs $D_{a,b,a'}$ on $n$ vertices, the minimum is
achieved for one with $a=\sqrt[4]{\sqrt 2\log\big(1+\sqrt 2\big)}\sqrt n+o(\sqrt n)$,
$a'=\sqrt[4]{\sqrt 2\log\big(1+\sqrt 2\big)}\sqrt n+o(\sqrt n)$ and $b=n-o(n)$.
\end{conjecture}

%
%

\section{Dumbbell-like graphs}

Dumbbell-like graph are obtained from dumbbell graphs by removing or attaching
some edges from or to the cliques.
As shown in~\cite{KKST}, they may have slightly smaller Balaban index, so we
are using the same approach with sum-Balaban index to see if this is true also
for sum-Balaban index.

We start with a precise definition.
A {\em dumbbell-like graph}, $D_{a,b,a'}^{\ell}$, is obtained from the dumbbell
graph $D_{a,b,a'}$ by either inserting ${\ell}$ edges between $v_1$ and
$K_a$ if $\ell>0$, or by removing  $-\ell$ edges between $v_{b-1}$ and
$K'_{a'}$ if $\ell<0$.
Note that we assume $a \le a'$, so we always add edges to the smaller
clique and remove them from the bigger one.
In~\cite{KKST} it was conjectured that dumbbell-like graphs
$D_{a,b,a'}^{\ell}$ attain the minimum value of Balaban index.
Experiments show that this happens in the cases $a=a'$ or $a=a'+1$,
$a\in \Theta(\sqrt{n})$.

This approach together with Conjecture~{\ref{C1}} suggests the following
two-step process for finding graphs with the minimum sum-Balaban index.

\begin{enumerate}
\item[$(i)$]
For a given $n$ find parameters $a$, $b$, $a'$, where $a+b+a'=n$ and
$a\le a'\le a+1$, such that $D_{a,b,a'}$ has the smallest sum-Balaban index.
\item[$(ii)$]
Find $\ell$ such that $D_{a,b,a'}^{\ell}$ has the smallest value of
sum-Balaban index.
\end{enumerate}

The outcome of the two-step process for $n$ satisfying $190\le n\le 210$
is presented in Table~\ref{table:comsim}.
In this table we present also the values of sum-Balaban index for optimal
dumbbell and dumbbell-like graphs. 

With the exception of $n = 13$, for all $n\le 210$ the smallest value
of sum-Balaban index for optimal dumbbell graph is obtained when
$a\le a'\le a+1$.
Case $n=13$ is special, since $D_{2,7,4}$ has indeed lower sum-Balaban index
than $D_{3,6,4}$.
But even in this case our two-step process finds the optimal dumbbell-like
graph on 13 vertices.

We conclude the paper with the following conjecture, which is supported
by our computer experiments.

\begin{conjecture}
\label{C3}
Dumbbell-like graphs $D_{a,b,a'}^{\ell}$  attain the minimum value of sum-Balaban index.
\end{conjecture}

%
%

For further recent topics and open problems in chemical graph theory an interested reader is referred to \cite{AKS-ful,dod1,dod2,KST,KSTsur,dod3}.
The quantitative graph measures are used also elsewhere, for example, in nowadays popular large networks, some results of the authors in that direction one can find in \cite{r5}.

\vskip 1pc
\noindent{\bf Acknowledgements.}~~Authors acknowledge partial support
by Slovak research grants VEGA 1/0007/14, VEGA 1/0026/16 and APVV 0136--12,
Slovenian research agency ARRS, program no.\ P1--0383, and  National Scholarship
Programme of the Slovak Republic SAIA.


\newpage

\section{Figures}


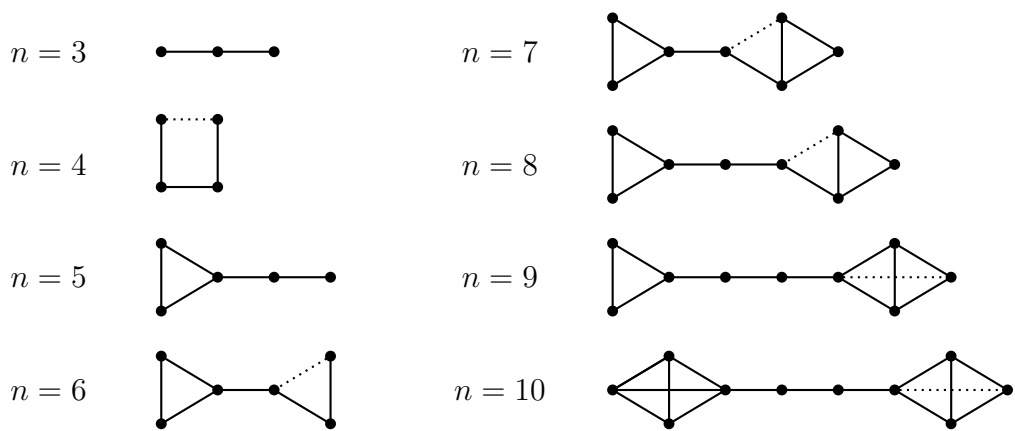
\begin{figure}[ht!]
\begin{center}
\begin{tikzpicture}[scale=1.5,style=thick]

\node[] at (-1,0) {$n=3$};
\node[] at (-1,-1) {$n=4$};
\node[] at (-1,-2) {$n=5$};
\node[] at (-1,-3) {$n=6$};

\node[] at (3,0) {$n=7$};
\node[] at (3,-1) {$n=8$};
\node[] at (3,-2) {$n=9$};
\node[] at (3,-3) {$n=10$};


\node [My Style, name=M] at (0,0) {};
\node [My Style, name=N]   at (0.5,0) {};
\node [My Style, name=O]   at (1,0) {};

\draw[] (M) -- (N);
\draw[] (N) -- (O);


\node [My Style, name=I] at (0,-0.6) {};
\node [My Style, name=J]   at (0,-1.2) {};
\node [My Style, name=K]   at (0.5,-1.2) {};
\node [My Style, name=L] at (0.5,-0.6) {};

\draw[] (I) -- (J);
\draw[] (J) -- (K);
\draw[] (K) -- (L);
\draw[dotted] (I) -- (L);


\node [My Style, name=i] at (0,-1.7) {};
\node [My Style, name=j]   at (0,-2.3) {};
\node [My Style, name=k]   at (0.5,-2) {};
\node [My Style, name=l] at (1,-2) {};
\node [My Style, name=m]   at (1.5,-2) {};

\draw[] (i) -- (j);
\draw[] (i) -- (k);
\draw[] (j) -- (k);
\draw[] (k) -- (l);
\draw[] (l) -- (m);


\node [My Style, name=P] at (0,-2.7) {};
\node [My Style, name=R]   at (0,-3.3) {};
\node [My Style, name=S]   at (0.5,-3) {};
\node [My Style, name=T] at (1,-3) {};
\node [My Style, name=U]   at (1.5,-3.3) {};
\node [My Style, name=V] at (1.5,-2.7) {};

\draw[] (P) -- (R);
\draw[] (R) -- (S);
\draw[] (S) -- (P);
\draw[] (S) -- (T);
\draw[] (T) -- (U);
\draw[] (U) -- (V);
\draw[dotted] (V) -- (T);


\node [My Style, name=A] at (4,0.3) {};
\node [My Style, name=B]   at (4,-0.3) {};
\node [My Style, name=C]   at (4.5,0) {};
\node [My Style, name=D] at (5,0) {};
\node [My Style, name=E]   at (5.5,-0.3) {};
\node [My Style, name=G] at (6,0) {};
\node [My Style, name=F] at (5.5,0.3) {};

\draw[] (A) -- (B);
\draw[] (A) -- (C);
\draw[] (B) -- (C);
\draw[] (C) -- (D);
\draw[] (D) -- (E);
\draw[] (E) -- (F);
\draw[] (E) -- (G);
\draw[] (F) -- (G);
\draw[dotted] (F) -- (D);


\node [My Style, name=a] at (4,-0.7) {};
\node [My Style, name=b]   at (4,-1.3) {};
\node [My Style, name=c]   at (4.5,-1) {};
\node [My Style, name=d] at (5,-1) {};
\node [My Style, name=e]   at (5.5,-1) {};
\node [My Style, name=f] at (6,-1.3) {};
\node [My Style, name=g] at (6,-0.7) {};
\node [My Style, name=h]   at (6.5,-1) {};

\draw[] (a) -- (b);
\draw[] (a) -- (c);
\draw[] (b) -- (c);
\draw[] (c) -- (d);
\draw[] (d) -- (e);
\draw[] (e) -- (f);
\draw[] (g) -- (f);
\draw[] (f) -- (h);
\draw[] (g) -- (h);
\draw[dotted] (g) -- (e);


\node [My Style, name=1]   at (4,-1.7) {};
\node [My Style, name=2]   at (4,-2.3) {};
\node [My Style, name=3] at (4.5,-2) {};
\node [My Style, name=4]   at (5,-2) {};
\node [My Style, name=5] at (5.5,-2) {};
\node [My Style, name=6]   at (6,-2) {};
\node [My Style, name=7]   at (6.5,-1.7) {};
\node [My Style, name=8] at (6.5,-2.3) {};
\node [My Style, name=9] at (7,-2) {};

\draw[] (1) -- (2);
\draw[] (2) -- (3);
\draw[] (1) -- (3);
\draw[] (3) -- (4);
\draw[] (4) -- (5);
\draw[] (5) -- (6);
\draw[] (6) -- (7);
\draw[] (6) -- (8);
\draw[] (7) -- (8);
\draw[] (9) -- (8);
\draw[] (7) -- (9);
\draw[dotted] (6) -- (9);


\node [My Style, name=o1]   at (4.5,-2.7) {};
\node [My Style, name=p1]   at (4,-3) {};
\node [My Style, name=r1] at (4.5,-3.3) {};
\node [My Style, name=s1]   at (5,-3) {};
\node [My Style, name=t1] at (5.5,-3) {};
\node [My Style, name=u1]   at (6,-3) {};
\node [My Style, name=v1]   at (6.5,-3) {};
\node [My Style, name=z1] at (7,-2.7) {};
\node [My Style, name=x1]   at (7,-3.3) {};
\node [My Style, name=y1]   at (7.5,-3) {};

\draw[] (o1) -- (p1);
\draw[] (o1) -- (r1);
\draw[] (r1) -- (p1);
\draw[] (o1) -- (p1);
\draw[] (s1) -- (p1);
\draw[] (r1) -- (s1);
\draw[] (s1) -- (t1);
\draw[] (t1) -- (u1);
\draw[] (u1) -- (v1);
\draw[] (v1) -- (z1);
\draw[] (v1) -- (x1);
\draw[dotted] (v1) -- (y1);
\draw[] (z1) -- (y1);
\draw[] (x1) -- (y1);
\draw[] (z1) -- (x1);
\draw[] (o1) -- (s1);

\end{tikzpicture}
\end{center}
\caption{Graphs with the smallest value of sum-Balaban index for $n\in \{3,4,\ldots,10\}$.}
\label{fig:dumb10}
\end{figure}


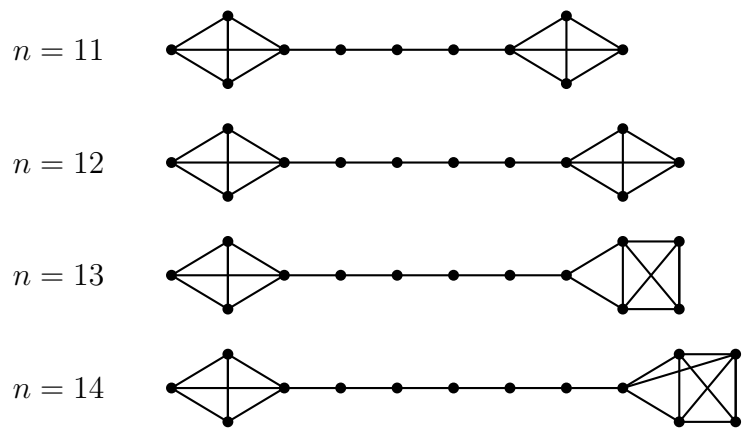
\begin{figure}[ht!]
\begin{center}
\begin{tikzpicture}[scale=1.5,style=thick]

\node[] at (-1,0) {$n=11$};
\node[] at (-1,-1) {$n=12$};
\node[] at (-1,-2) {$n=13$};
\node[] at (-1,-3) {$n=14$};

\node [My Style, name=n] at (0,0) {};
\node [My Style, name=o]   at (0.5,0.3) {};
\node [My Style, name=p]   at (0.5,-0.3) {};
\node [My Style, name=r] at (1,0) {};
\node [My Style, name=s]   at (1.5,0) {};
\node [My Style, name=t] at (2,0) {};
\node [My Style, name=u]   at (2.5,0) {};
\node [My Style, name=v]   at (3,0) {};
\node [My Style, name=z] at (3.5,0.3) {};
\node [My Style, name=x]   at (3.5,-0.3) {};
\node [My Style, name=y]   at (4,0) {};

\draw[] (n) -- (o);
\draw[] (o) -- (p);
\draw[] (p) -- (n);
\draw[] (o) -- (r);
\draw[] (r) -- (p);
\draw[] (o) -- (p);
\draw[] (r) -- (s);
\draw[] (s) -- (t);
\draw[] (t) -- (u);
\draw[] (u) -- (v);
\draw[] (v) -- (z);
\draw[] (v) -- (x);
\draw[] (v) -- (y);
\draw[] (z) -- (y);
\draw[] (x) -- (y);
\draw[] (z) -- (x);
\draw[] (r) -- (n);


\node [My Style, name=n1] at (0,-1) {};
\node [My Style, name=o1]   at (0.5,-0.7) {};
\node [My Style, name=p1]   at (0.5,-1.3) {};
\node [My Style, name=r1] at (1,-1) {};
\node [My Style, name=s1]   at (1.5,-1) {};
\node [My Style, name=t1] at (2,-1) {};
\node [My Style, name=u1]   at (3,-1) {};
\node [My Style, name=v1]   at (3.5,-1) {};
\node [My Style, name=z1] at (4,-0.7) {};
\node [My Style, name=x1]   at (4,-1.3) {};
\node [My Style, name=y1]   at (4.5,-1) {};
\node [My Style, name=a1]   at (2.5,-1) {};

\draw[] (n1) -- (o1);
\draw[] (o1) -- (p1);
\draw[] (p1) -- (n1);
\draw[] (o1) -- (r1);
\draw[] (r1) -- (p1);
\draw[] (o1) -- (p1);
\draw[] (r1) -- (s1);
\draw[] (s1) -- (t1);
\draw[] (t1) -- (a1);
\draw[] (u1) -- (a1);
\draw[] (u1) -- (v1);
\draw[] (v1) -- (z1);
\draw[] (v1) -- (x1);
\draw[] (v1) -- (y1);
\draw[] (z1) -- (y1);
\draw[] (x1) -- (y1);
\draw[] (z1) -- (x1);
\draw[] (r1) -- (n1);


\node [My Style, name=n2] at (0,-2) {};
\node [My Style, name=o2]   at (0.5,-1.7) {};
\node [My Style, name=p2]   at (0.5,-2.3) {};
\node [My Style, name=r2] at (1,-2) {};
\node [My Style, name=s2]   at (1.5,-2) {};
\node [My Style, name=t2] at (2,-2) {};
\node [My Style, name=u2]   at (2.5,-2) {};
\node [My Style, name=v2]   at (3,-2) {};

\node [My Style, name=z2] at (4,-1.7) {};
\node [My Style, name=x2]   at (4,-2.3) {};
\node [My Style, name=y2]   at (4.5,-1.7) {};
\node [My Style, name=w2]   at (4.5,-2.3) {};
\node [My Style, name=a2]   at (3.5,-2) {};

\draw[] (n2) -- (o2);
\draw[] (o2) -- (p2);
\draw[] (p2) -- (n2);
\draw[] (o2) -- (r2);
\draw[] (r2) -- (p2);
\draw[] (o2) -- (p2);
\draw[] (r2) -- (s2);
\draw[] (s2) -- (t2);
\draw[] (t2) -- (u2);
\draw[] (u2) -- (v2);
\draw[] (v2) -- (a2);
\draw[] (n2) -- (r2);

\draw[] (z2) -- (y2);
\draw[] (z2) -- (a2);
\draw[] (w2) -- (y2);
\draw[] (w2) -- (x2);
\draw[] (x2) -- (a2);
\draw[] (x2) -- (z2);
\draw[] (y2) -- (w2);
\draw[] (x2) -- (y2);
\draw[] (z2) -- (w2);


\node [My Style, name=n3] at (0,-3) {};
\node [My Style, name=o3]   at (0.5,-2.7) {};
\node [My Style, name=p3]   at (0.5,-3.3) {};
\node [My Style, name=r3] at (1,-3) {};
\node [My Style, name=s3]   at (1.5,-3) {};
\node [My Style, name=t3] at (2,-3) {};
\node [My Style, name=u3]   at (2.5,-3) {};
\node [My Style, name=v3]   at (3,-3) {};

\node [My Style, name=z3] at (4.5,-2.7) {};
\node [My Style, name=x3]   at (4.5,-3.3) {};
\node [My Style, name=y3]   at (5,-2.7) {};
\node [My Style, name=w3]   at (5,-3.3) {};
\node [My Style, name=a3]   at (4,-3) {};

\node [My Style, name=b3]   at (3.5,-3) {};

\draw[] (n3) -- (o3);
\draw[] (o3) -- (p3);
\draw[] (p3) -- (n3);
\draw[] (o3) -- (r3);
\draw[] (r3) -- (p3);
\draw[] (o3) -- (p3);
\draw[] (r3) -- (s3);
\draw[] (s3) -- (t3);
\draw[] (t3) -- (u3);
\draw[] (u3) -- (v3);
\draw[] (v3) -- (b3);
\draw[] (b3) -- (a3);
\draw[] (n3) -- (r3);

\draw[] (z3) -- (y3);
\draw[] (z3) -- (a3);
\draw[] (w3) -- (y3);
\draw[] (w3) -- (x3);
\draw[] (x3) -- (a3);
\draw[] (x3) -- (z3);
\draw[] (y3) -- (w3);
\draw[] (x3) -- (y3);
\draw[] (z3) -- (w3);

\draw[] (a3) -- (y3);
\end{tikzpicture}
\end{center}
\caption{Candidates for the smallest value of sum-Balaban index for
$n\in\{11,12,13,14\}$.
}
\label{fig:dumb11}
\end{figure}


\newpage 

\section{Tables}

\begin{table}[!ht]
\centering
\begin{tabular}{|c|c|c|c|c|c|c|}
\hline
$n$&$a$&$a'$&$b$&$J(D_{a,b,a'})$&$\ell$&$J(D_{a,b,a'}^{\ell})$\\\hline

$190$ & $14$ & $15$ & $161$ & $4.6411$ &
$-5$ & $4.6405$ \\\hline
$191$ & $14$ & $15$ & $162$ & $4.6405$ &
$-3$ & $4.6401$ \\\hline
$192$ & $14$ & $15$ & $163$ & $4.6399$ &
$-2$ & $4.6397$ \\\hline
$193$ & $14$ & $15$ & $164$ & $4.6394$ &
$-1$ & $4.6393$ \\\hline
$194$ & $14$ & $15$ & $165$ & $4.6389$ &
$0$ & $4.6389$ \\\hline
$195$ & $14$ & $15$ & $166$ & $4.6386$ & $1$ &
$4.6385$ \\\hline
$196$ & $14$ & $15$ & $167$ & $4.6383$ & $2$ &
$4.6381$ \\\hline
$197$ & $14$ & $15$ & $168$ & $4.6381$ & $3$ &
$4.6377$ \\\hline
$198$ & $14$ & $15$ & $169$ & $4.6379$ & $4$ &
$4.6373$ \\\hline
$199$ & $14$ & $15$ & $170$ & $4.6379$ & $6$ &
$4.6369$ \\\hline
$200$ & $14$ & $15$ & $171$ & $4.6379$ & $7$ &
$4.6365$ \\\hline
$201$ & $15$ & $15$ & $171$ & $4.6372$ & 
$-6$ & $4.6361$ \\\hline
$202$ & $15$ & $15$ & $172$ & $4.6364$ & 
$-5$ & $4.6357$ \\\hline
$203$ & $15$ & $15$ & $173$ & $4.6357$ &
$-4$ & $4.6353$ \\\hline
$204$ & $15$ & $15$ & $174$ & $4.6351$ &
$-3$ & $4.6349$ \\\hline
$205$ & $15$ & $15$ & $175$ & $4.6346$ &
$-2$ & $4.6345$ \\\hline
$206$ & $15$ & $15$ & $176$ & $4.6341$ &
$0$ & $4.6341$ \\\hline
$207$ & $15$ & $15$ & $177$ & $4.6337$ &
$0$ & $4.6337$ \\\hline
$208$ & $15$ & $15$ & $178$ & $4.6334$ &
$0$ & $4.6334$ \\\hline
$209$ & $15$ & $15$ & $179$ & $4.6331$ & $1$ &
$4.6331$ \\\hline
$210$ & $15$ & $15$ & $180$ & $4.6329$ & $3$ &
$4.6328$ \\\hline
\end{tabular}
\caption{The optimal dumbbell graph $D_{a,b,a'}$ on $n\in [190,210]$ vertices,
and its improvement $D_{a,b,a'}^{\ell}$ in the class of dumbbell-like graphs.}
\label{table:comsim}
\end{table}

\end{document}